\documentclass{amsart}
\usepackage{stmaryrd,mathabx,xcolor}

\usepackage[pagebackref,colorlinks=true
]{hyperref}  

 \newtheorem{thm}{Theorem}[section]
 \newtheorem{cor}[thm]{Corollary}
 \newtheorem{lem}[thm]{Lemma}
 \newtheorem{prop}[thm]{Proposition}
 \theoremstyle{definition}
 
 \theoremstyle{remark}

 \numberwithin{equation}{section}

\newcommand{\W}{\mathcal{W}}
\newcommand{\R}{\mathbb{R}}

\newcommand{\n}{\nabla}
\newcommand{\tr}{\mathrm{tr}}
\newcommand{\tg}{\widetilde{g}}
\newcommand{\tn}{\widetilde\nabla}
\newcommand{\tR}{\widetilde{R}}
\newcommand{\tF}{\widetilde{F}}

\newcommand{\tQ}{\widetilde{Q}}

\newcommand{\tB}{\widetilde{B}}
\newcommand{\tP}{\widetilde{\Phi}}
\newcommand{\tN}{\widetilde{N}}
\newcommand{\ta}{\theta}
\newcommand{\lm}{\lambda}
\newcommand{\e}{\varepsilon}
\newcommand{\D}{\mathrm{d}}

\newcommand{\norm}[1]{\Vert#1\Vert ^2}
\newcommand{\nP}{\norm{\nabla P}}
\newcommand{\tnP}{\norm{\widetilde{\nabla} P}}
\newcommand{\wh}[1]{\widehat{#1}}

\newcommand{\thmref}[1]{Theorem~\ref{#1}}
\newcommand{\propref}[1]{Proposition~\ref{#1}}
\newcommand{\lemref}[1]{Lemma~\ref{#1}}

\begin{document}
%
%
%
%
%
\title[Invariants on Almost Paracomplex Pseudo-Riemannian Manifolds]
 {Invariant Tensors under the Twin Interchange\\
of the Pairs of the Associated Metrics
on Almost Paracomplex Pseudo-Riemannian Manifolds}
\author[M. Manev]{Mancho Manev}

\address{%
  Department of Algebra and Geometry \\
  Faculty of Mathematics and Informatics \\
  University of Plovdiv Paisii Hilendarski  \\
  24 Tzar Asen St \\
  4000 Plovdiv, Bulgaria\\
	\& \\
	Department of Medical Informatics, Biostatistics and E-Learning\\
	Faculty of Public Health\\
	Medical University of Plovdiv \\
	15A 	Vasil Aprilov Blvd\\
	4002 Plovdiv, Bulgaria
}

\email{mmanev@uni-plovdiv.bg}

\subjclass[2010]{Primary 53C15; Secondary 53C25}

\keywords{Invariant tensor, affine connection, almost paracomplex manifold, pseudo-Riemannian metric}


\begin{abstract}
The object of study are almost paracomplex pseudo-Riemannian manifolds with a pair of metrics associated each other by the almost paracomplex structure. A torsion-free connection and tensors with geometric interpretation are found which are invariant under the twin interchange, i.e. the swap of the counterparts of the pair of associated metrics and the corresponding Levi-Civita connections. A Lie group depending on two real parameters is constructed as an example of a 4-dimensional manifold of the studied type and the mentioned invariant objects are found in an explicit form.
\end{abstract}

\maketitle

\section*{Introduction}

Manifolds with almost product structure and Riemannian metric are well known \cite{Nav}. Usually, the almost product structure acts as an isometry with respect to the metric, i.e. it is said that the metric is \emph{compatible} with the structure. A special and remarkable case is when the almost product structure is traceless and then it is called an \emph{almost paracomplex structure}. In this case, the eigenvalues $+1$ and $-1$ of the structure have one and the same multiplicity, thus the dimension of such a manifold is even. An almost paracomplex manifold is a counterpart of an almost complex manifold. The compatible metric with an almost complex structure is a Hermitian metric. The requirement that the metric be Riemannian on an almost paracomplex manifold is not necessarily and thus we suppose here that the metric is pseudo-Riemannian.

The associated $(0,2)$-tensor of a Hermitian metric is a 2-form while the associated $(0,2)$-tensor of any compatible metric on almost paracomplex manifold is also a compatible metric. So, in this case, we dispose of a pair of mutually associated compatible metrics with respect to the almost paracomplex structure, known also as twin metrics. Such \emph{almost paracomplex manifolds} are studied in the latter three decades by a lot of authors (e.g. 
\cite{Sta87},  \cite{Mih87}, \cite{StaGri92}, 
\cite{BoFeFrVo99}, \cite{Pri05}, \cite{Dok05},
\cite{Mek08}, \cite{DGriMek11}, \cite{MekMan12}, \cite{DGri12}), 
including under the name \emph{Riemannian almost product manifolds}.

An interesting problem on almost paracomplex (pseudo-)Riemannian manifolds is the presence of tensors with some geometric interpretation which are invariant or anti-invariant under the so-called \emph{twin interchange}. This is the swap of the counterparts of the pair of compatible metrics and their Levi-Civita connections. The aim of the present work is to solve this problem in the general case and to illustrate the invariant objects by example from a significant class of the considered manifolds.
Invariant connection and invariant tensors under twin interchange on Riemannian almost product manifolds with nonintegrable structure are found in \cite{Mek08}. 
A similar investigation for almost complex manifolds with Norden metrics is given in \cite{Man49}. An explicit example of a Riemannian almost product manifold from the main class is proposed in \cite{DGri12}. Similar investigations on Lie groups with additional tensor structures are made in \cite{GriManMek06} and \cite{HMMek15}.

The present paper is organized as follows. Section~1 contains some preliminaries on the considered type of manifolds. In Section~2 we present the main results on the topic about the invariant objects and their vanishing. In Section~3 we give a specialization of the considered tensors when the manifolds under study belongs to the main class. In Section~4 we construct an example of the studied manifolds of dimension 4 by an appropriate Lie algebra depending on 2 real parameters. Then we compute the basic components of the invariant objects which are found in the previous sections.

\section{Almost paracomplex pseudo-Riemannian manifolds}\label{sec_1}



Let $(M,P,g)$ be an almost paracomplex pseudo-Riemannian manifold.
Consequently,  $P$ is an almost paracomplex structure, i.e.
\begin{equation}\label{str-P}
P^2x=x,\qquad \tr P = 0,
\end{equation}
and $g$ is
a pseudo-Riemannian metric on $M$ compatible with $P$, i.e.
\begin{equation}\label{str-g}
g(Px,Py)=g(x,y).
\end{equation}
Here and further, $x$, $y$, $z$, $w$
will stand for arbitrary differentiable vector fields on $M$
or vectors in $T_pM$, $p\in M$.

In the present work, $(M,P,g)$ is called briefly a \emph{$\W$-manifold} and $g$ -- a \emph{$P$-metric}.

Necessarily, the dimension of this manifold is even, i.e. $\dim M=2n$, $n\in\mathbb{N}$.
Then the signature of $g$ is of the type $(2q,2n-2q)$ for some fixed $q\in\{1,2,\dots,n\}$.

Let $\{e_i\}$ ($i=1,2,\dots,2n$) be a basis of $T_pM$ at a point $p$ of $M$ such that
\begin{equation*}\label{ei}
\begin{array}{ll}
g(e_k,e_k)=1, \quad & k\in\{1,2,\dots,2q\},\\[4pt]
g(e_l,e_l)=-1, \quad & l\in\{2q+1,2q+2,\dots,2n\},\\[4pt]
g(e_i,e_j)=0, \quad  & i\neq j,\quad i,j\in\{1,2,\dots,2n\}.
\end{array}
\end{equation*}
An almost paracomplex structure $P$ is defined as follows
\begin{equation*}\label{Pei}
\begin{array}{lll}
Pe_{2k-1}=e_{2k},\quad & Pe_{2k}=e_{2k-1}, \quad & k\in\{1,2,\dots,q\};\\[4pt]
Pe_{2l-1}=e_{2l},\quad   & Pe_{2l}=e_{2l-1}, \quad   & l\in\{q+1,q+2,\dots,n\}.
\end{array}
\end{equation*}

Obviously, this basis is compatible with $P$ because \eqref{str-g} is satisfied.
Let us call the introduced basis an \emph{adapted $P$-basis} of the considered $\W$-manifold.

Let us consider the basis $\{a_i\}$ ($i=1,2,\dots,2n$) defined by the basis $\{e_i\}$ as follows
\begin{equation}\label{eigen}
\begin{array}{ll}
a_{2k-1}=\frac{1}{\sqrt{2}}(e_{2k-1}-e_{2k}),\quad & a_{2k}=\frac{1}{\sqrt{2}}(e_{2k-1}+e_{2k}), \\[4pt]
a_{2l-1}=\frac{1}{\sqrt{2}}(e_{2l-1}-e_{2l}),\quad & a_{2l}=\frac{1}{\sqrt{2}}(e_{2l-1}+e_{2l})
\end{array}
\end{equation}
for $k\in\{1,2,\dots,q\}$, $l\in\{q+1,q+2,\dots,n\}$.
These vectors satisfy the following equalities
\begin{equation*}\label{eigenP}
\begin{array}{llll}
Pa_{2k-1}=-a_{2k-1},\quad & Pa_{2k}=a_{2k}, \quad & 
Pa_{2l-1}=-a_{2l-1},\quad & Pa_{2l}=a_{2l}
\end{array}
\end{equation*}
and therefore $\{a_i\}$ is an \emph{eigenbasis} with respect to $P$ of $T_pM$.

For an arbitrary $\W$-manifold $(M,P,g)$, there exists an associated metric $\tg$ of $g$ given by
\begin{equation}\label{tgP}
\tg(x,y)=g(x,Py).
\end{equation}
It is also a $P$-metric since obviously the condition $\tg(Px,Py)=\tg(x,y)$ is satisfied.
By virtue of the following equalities for the vectors of $\{a_i\}$ from \eqref{eigen}
\begin{equation*}\label{tg-eigenP}
\begin{array}{ll}
\tg(a_{2k-1},a_{2k-1})=-1,\quad & \tg(a_{2k},a_{2k})=1, \\[4pt]
\tg(a_{2l-1},a_{2l-1})=-1,\quad & \tg(a_{2l},a_{2l})=1, 
\end{array}
\end{equation*}
we conclude that
the signature of $\tg$ is $(n,n)$.

Together with \eqref{tgP}, the relation $\tg(x,Py)=g(x,y)$ is also valid. Thus, these metrics we call \emph{twin $P$-metrics} on $M$.

The Levi-Civita connections of $g$ and $\tg$ are denoted by $\n$ and $\tn$, respectively.
The interchange of $\n$ and $\tn$ (and respectively $g$ and $\tg$) we call the \emph{twin interchange}.

The
tensor filed $F$ of type $(0,3)$ on $M$ is defined by
\begin{equation}\label{F}
F(x,y,z)=g\bigl( \left( \nabla_x P \right)y,z\bigr).
\end{equation}
It has the following properties \cite{Nav}
\begin{equation}\label{F-prop}
F(x,y,z)=F(x,z,y)=-F(x,Py,Pz),\quad F(x,Py,z)=-F(x,y,Pz).
\end{equation}

Let $\{e_i\}$ ($i=1,2,\dots,2n$) be an arbitrary basis of
$T_pM$ at a point $p$ of $M$. The components of the inverse matrix
of $g$ are denoted by $g^{ij}$ with respect to
$\{e_i\}$.

The Lee forms $\ta$ and $\ta^*$ associated with $F$ are defined by
\begin{equation}\label{ta}
\ta(z)=g^{ij}F(e_i,e_j,z),\quad \ta^*(z)=g^{ij}F(e_i,Pe_j,z).
\end{equation}
For the 1-form $\ta^*$, using $\tg$, we have the following
\[
\ta^*(z)=g^{ij}F(e_i,Pe_j,z)=P^j_kg^{ik}F(e_i,e_j,z)=\tg^{ij}F(e_i,e_j,z)
\]
and the identity
\begin{equation}\label{ta*taJ}
\ta^*=-\ta\circ P
\end{equation}
is true by means of \eqref{F-prop} due to
\[
\ta^*(z)=g^{ij}F(e_i,Pe_j,z)=-g^{ij}F(e_i,e_j,Pz)=-\ta(Pz).
\]

The \emph{potential} $\Phi$ of $\tn$ regarding $\n$ is given by the formula
\begin{equation}\label{tn=nPhi}
    \tn_x y=\n_x y+\Phi(x,y).
\end{equation}
Since both the connections are torsion-free, then $\Phi$ is symmetric, i.e. $\Phi(x,y)=\Phi(y,x)$.
Let the corresponding tensor of type $(0,3)$ with respect to $g$ be defined by
\begin{equation}\label{Phi03}
    \Phi(x,y,z)=g(\Phi(x,y),z).
\end{equation}

     By virtue of \eqref{F-prop} and \eqref{tn=nPhi}, the following interrelations between $F$ and $\Phi$ are valid \cite{StaGri92}
\begin{equation}\label{PhiF}
    \Phi(x,y,z)=\frac{1}{2}\bigl\{F(x,y,Pz)+F(y,x,Pz)-F(Pz,x,y)\bigr\},
\end{equation}
\begin{equation}\label{FPhi}
    F(x,y,z)=\Phi(x,y,Pz)+\Phi(x,z,Py).
\end{equation}

Taking into account \eqref{F-prop} and \eqref{FPhi}, we obtain the following property
for an arbitrary $\W$-manifold
\begin{equation}\label{Phi-prop}
  \Phi(x,y,z)+\Phi(x,z,y)+\Phi(x,Py,Pz)+\Phi(x,Pz,Py)=0.
\end{equation}

The associated 1-forms $f$ and $f^*$ of $\Phi$ are defined by
\[
f(z)=g^{ij}\Phi(e_i,e_j,z),\qquad f^*(z)=g^{ij}\Phi(e_i,Pe_j,z).
\]
Then, from \eqref{Phi-prop} we get the identity
\begin{equation}\label{f-prop}
f=-f^*\circ P.
\end{equation}
The latter identity resembles the equality $\ta=-\ta^*\circ P$, which is equivalent to \eqref{ta*taJ}.
Indeed, there exists a relation between the associated 1-forms of $F$ and $\Phi$.
It follows from \eqref{PhiF}, \eqref{ta*taJ} and has the form
\begin{equation}\label{fta}
f=-\ta^*,\qquad f^*=-\ta.
\end{equation}

A classification of Riemannian almost product  manifolds having a traceless structure $P$ with respect to $F$
is given in \cite{StaGri92}.
It is applicable to the considered $\W$-manifolds.
All eight classes of these manifolds are characterized there by properties of $F$ as follows
\begin{equation}\label{class}
\begin{array}{l}
\W_0:\; F(x,y,z)=0;\\[4pt]
\W_1:\; F(x,y,z)=\frac{1}{2n} \bigl\{
g(x,y)\ta(z)+\tg(x, y)\ta^*(z)\\[4pt]
\phantom{\W_1:\; F(x,y,z)=\frac{1}{2n}}
    +g(x,z)\ta(y)+\tg(x,z)\ta^*(y)\bigr\};\\[4pt]
\W_2:\;
F(x,y,P z)+F(y,z,P x)+F(z,x,P y)=0,\quad \ta=0;\\[4pt]
\W_3:\; F(x,y,z)+F(y,z,x)+F(z,x,y)=0;\\[4pt]
\W_1\oplus\W_2:\;
F(x,y,P z)+F(y,z,P x)+F(z,x,P y)=0;\\[4pt]
\W_1\oplus\W_3:\;
F(x,y,z)+F(y,z,x)+F(z,x,y)\\[4pt]
\phantom{\W_1\oplus\W_3:\;}
=\frac{1}{n} \bigl\{
g(x,y)\ta(z)+g(y,z)\ta(x)+g(z,x)\ta(y)\\[4pt]
\phantom{\W_1\oplus\W_3:\; =\frac{1}{n}\;\;}
    +\tg(x,y)\ta^*(z)+\tg(y,z)\ta^*(x)+\tg(z,x)\ta^*(y)\bigr\};\\[4pt]
\W_2\oplus\W_3:\;
\ta=0;\\[4pt]
\W_1\oplus\W_2\oplus\W_3:\;
\text{no conditions}.
\end{array}
\end{equation}
An equivalent classification in terms of $\Phi$ is proposed in the same paper by the following way
\begin{equation}\label{class2}
\begin{array}{l}
\W_0:\; \Phi(x,y,z)=0;\\[4pt]
\W_1:\; \Phi(x,y,z)=\frac{1}{2n}\left\{g(x,y)f(z)+\tg(x,y)f^*(z)\right\};\\[4pt]
\W_2:\;
\Phi(x,y,z)=\Phi(Px,Py,z),\quad f=0;\\[4pt]
\W_3:\; \Phi(x,y,z)=-\Phi(Px,Py,z);\\[4pt]
\W_1\oplus\W_2:\;
\Phi(x,y,z)=\Phi(Px,Py,z);\\[4pt]
\W_1\oplus\W_3:\;
\Phi(x,y,z)+\Phi(Px,Py,z)
=\frac{1}{n} \bigl\{
g(x,y)f(z)+\tg(x,y)f^*(z)\bigr\};\\[4pt]
\W_2\oplus\W_3:\;
f=0;\\[4pt]
\W_1\oplus\W_2\oplus\W_3:\;
\text{no conditions}.
\end{array}
\end{equation}

The square norm of $\nabla P$ is defined by the following equality \cite{Mek08}
\[
    \nP=g^{ij}g^{kl}
        g\bigl(\left(\nabla_{e_i} P\right)e_k,\left(\nabla_{e_j} P\right)e_l\bigr).
\]
By means of \eqref{F} and \eqref{F-prop}, we obtain the following equivalent formula
\begin{equation}\label{snorm}
    \nP=g^{ij}g^{kl}g^{st}F_{iks}F_{jlt},
\end{equation}
where $F_{iks}=F(e_i,e_k,e_s)$.
A $\W$-manifold satisfying the
condition $\nP=0$ we call an
\emph{isotropic $\W_0$-manifold}.
Obviously, if a $\W$-manifold belongs to $\W_0$ (i.e. it is a $\W_0$-manifold), then
it is an isotropic $\W_0$-manifold. Let us remark that the inverse statement is not
always true.


Let $R$ be the curvature tensor field of $\nabla$ defined by
$
    R(x,y)z=\nabla_x \nabla_y z - \nabla_y \nabla_x z -
    \nabla_{[x,y]}z$.
The corresponding tensor field of type $(0,4)$ is determined by
$R(x,y,z,w)=g(R(x,y)z,w)$. It has the following properties:
\begin{equation}\label{curv}
\begin{array}{l}%
R(x,y,z,w)=-R(y,x,z,w)=-R(x,y,w,z),\\[4pt]
R(x,y,z,w)+R(y,z,x,w)+R(z,x,y,w)=0.
\end{array}%
\end{equation}
Any tensor of type (0,4) satisfying \eqref{curv}
is called a \emph{curvature-like tensor}.
The Ricci tensor $\rho$ and the scalar
curvature $\tau$ for $R$ 
are defined as usual by the equalities 
$\rho(y,z)=g^{ij}R(e_i,y,z,e_j)$ and $\tau=g^{ij}\rho(e_i,e_j)$.
%
%

Let $\tR$ be the curvature tensor of $\tn$ defined as usually.
Obviously, the corresponding curvature (0,4)-tensor is
$\tR(x,y,z,u)=\tg(\tR(x,y)z,u)$ and it has the same properties as in \eqref{curv}.


\section{The twin interchange corresponding to the pair of twin $P$-metrics and their Levi-Civita connections}



\subsection{Invariant classification}

\begin{lem}\label{lem:Phi}
  The potential $\Phi(x,y)$ is an anti-invariant tensor under the twin interchange, i.e.
\begin{equation}\label{tP=-P}
    \tP(x,y)=-\Phi(x,y).
\end{equation}
\end{lem}
\begin{proof}
Taking into account \eqref{F-prop},
\eqref{tn=nPhi},  \eqref{Phi03} and \eqref{PhiF}, we get the following relation between $F$ and its corresponding tensor $\tF$ for $(M,P,\tg)$, defined by $\tF(x,y,z)=\tg\bigl(\bigl(\tn_xP\bigr)y,z\bigr)$,
\begin{equation}\label{tFF}
    \tF(x,y,z)=\frac{1}{2}\bigl\{F(Py,z,x)-F(y,Pz,x)+F(Pz,y,x)-F(z,Py,x)\bigr\}.
\end{equation}
Applying \eqref{PhiF}, we obtain the corresponding formula for $\tP$ and $\tF$ in the form
\begin{equation}\label{tPtF}
				\tP(x,y,z)=\frac{1}{2}\bigl\{\tF(x,y,Pz)+\tF(y,x,Pz)-\tF(Pz,x,y)\bigr\}.
\end{equation}
Using \eqref{tFF}, \eqref{tPtF} and \eqref{PhiF}, we reach the following equality
\begin{equation}\label{tPPhi}
\tP(x,y,z)=-\Phi(x,y,Pz).
\end{equation}
The identity \eqref{tP=-P} follows from the latter eqiality and the definition of $\tP(x,y,z)$ by $\tP(x,y,z)=\tg(\tP(x,y),z)$.
\end{proof}

\begin{lem}\label{lem:f}
The associated 1-forms $f$ and $f^*$ of $\Phi$ and the Lee forms $\ta$ and $\ta^*$ are invariant under the twin interchange, i.e.
\begin{gather}
\widetilde{f}(z)=f(z), \qquad \widetilde{f}^*(z)=f^*(z);\label{tff}
\\
\ta(z)=\widetilde{\ta}(z),\qquad \ta^*(z)=\widetilde{\ta}^*(z).\label{tata*}
\end{gather}
\end{lem}
\begin{proof}
Taking the trace of \eqref{tPPhi} by $\tg^{ij}=P^j_kg^{ik}$ for $x=e_i$ and $y=e_j$, we have $\widetilde{f}(z)=-f^*(Pz)$. Then, because of \eqref{f-prop}, we obtain the statement for $f$ and similarly for $f^*$. The invariance of the Lee forms follows directly from \eqref{tff} and \eqref{fta}.
\end{proof}

\begin{thm}\label{thm:inv.cl}
  All eight classes of almost paracomplex pseudo-Riemannian manifolds are invariant under the twin interchange.
\end{thm}
\begin{proof}
We use the classification \eqref{class2} in terms of $\Phi$.
Obviously, applying \lemref{lem:Phi}, \lemref{lem:f}, equalities \eqref{tPPhi}, \eqref{Phi03} and \eqref{f-prop}, we establish the truthfulness of the statement.
\end{proof}


\subsection{Invariant connection}

Let us define an affine connection $D$ on $TM$ by
\begin{equation}\label{hn=nP}
D_x y=\n_x y+\frac12\Phi(x,y).
\end{equation}
Applying \eqref{tn=nPhi}, we find that $D$ is actually the \emph{average connection} of $\n$ and $\tn$, i.e.
\begin{equation}\label{av.con}
D_x y=
\frac12\left\{\n_x y+\tn_x y\right\}.
\end{equation}

\begin{prop}\label{prop:inv.conn}
    The average connection $D$ of $\n$ and $\tn$ is an invariant connection under the twin interchange.
\end{prop}
\begin{proof}
It follows from \eqref{tn=nPhi}, \eqref{tP=-P} and \eqref{hn=nP}, because of the following equalities
\[
\widetilde{D}_x y=\tn_x y+\frac12\tP(x,y)=\n_x y+\Phi(x,y)-\frac12\Phi(x,y)
=\n_x y+\frac12\Phi(x,y)=D_x y.
\]
\end{proof}

\begin{cor}\label{cor:inv.conn}
    If the invariant connection $D$ vanishes then $(M,P,g)$ and $(M,P,\tg)$ are $\W_0$-manifolds and the coincidental connections $\n$ and $\tn$ also vanish.
\end{cor}
\begin{proof}
If $D=0$ holds, then we have $\n=-\tn$ and $\Phi=-2\n$, because of \eqref{hn=nP} and \eqref{av.con}. Hence we obtain $[x,y]=\n_x y-\n_y x=-\frac12\{\Phi(x,y)-\Phi(y,x)\}=0$ and consequently, using the Koszul formula
\[
2g(\nabla_x y,z)=xg(y,z)+yg(x,z)-zg(x,y)+g([x,y],z)+g([z,x],y)+g([z,y],x),
\]
we get $\n=0$. Thus, $\tn$ and $\Phi$ vanish, i.e. $(M,P,g)$ and $(M,P,\tg)$ are $\W_0$-manifolds.
\end{proof}

\subsection{Invariant tensors}

The Nijenhuis tensor $N$ of the almost paracomplex structure $P$ is
defined by
\begin{equation}\label{NJ}
N(x,y) = [P, P](x, y)=\left[Px,Py\right]+\left[x,y\right]-P\left[Px,y\right]-P\left[x,Py\right].
\end{equation}
In \cite{ManTav19}, it is introduced a symmetric
(1,2)-tensor $\wh{N}$, defined by
\begin{equation*}\label{tNJ}
\wh{N}(x,y)=\{P ,P\}(x,y)=\{Px,Py\}+\{x,y\}-P\{Px,y\}-P\{x,Py\},
\end{equation*}
where the symmetric braces $\{x,y\}=\nabla_xy+\nabla_yx$
replace the antisymmetric brackets $[x,y]=\nabla_xy-\nabla_yx$ in \eqref{NJ}.
The tensor $\wh{N}$ is called the  \emph{associated
Nijenhuis tensor} of $P$. The tensor $\wh{N}$ coincides with the associated tensor of ${N}$ introduced in \cite{StaGri92} by an equivalent equality for $F$.
The corresponding tensors of type $(0,3)$ with respect to $g$ of the pair of Nijenhuis tensors $N$ and $\tN$ are defined by
\begin{equation}\label{NtN3}
N(x,y,z)=g\bigl(N(x,y),z\bigr),\qquad \wh{N}(x,y,z)=g\bigl(\wh{N}(x,y),z\bigr).
\end{equation}


\begin{prop}\label{prop:inv.NhN}
    The Nijenhuis tensor is invariant and the associated Nijenhuis tensor is anti-invariant under the twin interchange, i.e.
    \[
    N(x,y)=\tN(x,y),\qquad \wh{N}(x,y)=-\widetilde{\wh{N}}(x,y).
    \]
\end{prop}
\begin{proof}
The following relations of $N$ and $\wh{N}$ with $\Phi$ are known from \cite{StaGri92} 
\begin{gather}
  N(x,y,z)=2\Phi(z,x,y)+2\Phi(z,Px,Py),\label{NPhi}\nonumber\\
  \wh{N}(x,y,z)=-2\Phi(x,y,z)-2\Phi(Px,Py,z).\label{wNPhi}
\end{gather}
Taking into account \eqref{tPPhi}, the latter equalities imply the following
\begin{gather}
  \tN(x,y,z)=-N(x,Py,z),\label{tNN}\\
  \widetilde{\wh{N}}(x,y,z)=-\wh{N}(x,y,Pz).\label{twNwN}
\end{gather}
By virtue of \eqref{NJ}, \eqref{NtN3}, \eqref{str-P} and \eqref{str-g}, we establish the truthfulness of  the property
$N(x, y,z) = -N(x, Py, Pz)$ which is equivalent to $N(x,Py,z) = -N(x,y, Pz)$. Then \eqref{tNN} gets the form
\begin{gather}
  \tN(x,y,z)=N(x,y,Pz).\label{tNN2}
\end{gather}
From \eqref{tNN2} and \eqref{twNwN} we obtain the relations in the statement.
\end{proof}

The following relation between the curvature tensors of $\n$ and $\tn$ related by \eqref{tn=nPhi} is well-known 
\begin{equation}\label{tRRS}
    \tR(x,y)z=R(x,y)z+Q(x,y)z,
\end{equation}
where
\begin{equation}\label{Q}
    Q(x,y)z= \left(\n_x \Phi\right)(y,z)- \left(\n_y \Phi\right)(x,z)
    +\Phi\left(x,\Phi(y,z)\right)-\Phi\left(y,\Phi(x,z)\right).
\end{equation}

The following tensor is a part of the tensor $Q$
\begin{equation}\label{B}
B(x,y)z=\Phi(x,\Phi(y,z))-\Phi(y,\Phi(x,z)).
\end{equation}
\begin{lem}\label{lem:B}
  The tensor $B(x,y)z$ is invariant under the twin interchange, i.e.
\begin{equation}\label{B=tA}
B(x,y)z=\tB(x,y)z.
\end{equation}
\end{lem}
\begin{proof}
It follows directly from from \eqref{tP=-P} and \eqref{B}.
\end{proof}

\begin{lem}\label{lem:Q}
  The tensor $Q(x,y)z$ is anti-invariant under the twin interchange, i.e.
\begin{equation}\label{tS=-Q}
    \tQ(x,y)z=-Q(x,y)z.
\end{equation}
\end{lem}
\begin{proof}
Applying \eqref{tn=nPhi} and \eqref{tP=-P} to the formula for the covariant derivative of $\tP$ with respect to $\tn$, i.e.
$
\left(\tn_x\tP\right)(y,z)=\tn_x\tP(y,z)-\tP(\n_x y,z)-\tP(y,\n_x z),
$ we get
\[
(\tn_x\tP)(y,z)=-\left(\n_x\Phi\right)(y,z)-\Phi(x,\Phi(y,z))+\Phi(y,\Phi(x,z))+\Phi(z,\Phi(x,y)).
\]
The latter equality and \eqref{B} yield
\begin{equation}\label{nP}
\begin{array}{l}
    (\tn_x\tP)(y,z)-(\tn_y\tP)(x,z)=-\left(\n_x\Phi\right)(y,z)+\left(\n_y\Phi\right)(x,z)
    -2{B}(x,y)z.
\end{array}
\end{equation}
Then, equalities \eqref{Q}, \eqref{B}, \eqref{B=tA} and \eqref{nP} imply \eqref{tS=-Q}.
\end{proof}

\begin{prop}\label{prop:inv.tensor2}
    The curvature tensor $K$ of the average connection $D$ for $\n$ and $\tn$ is an invariant tensor under the twin interchange, i.e. $K(x,y)z=\widetilde{K}(x,y)z$.
\end{prop}
\begin{proof}
By virtue of \eqref{hn=nP}, \eqref{tn=nPhi}, \eqref{tRRS}, \eqref{Q} and \eqref{B}, we get the formula 
\begin{equation}\label{hRRSA}
    K(x,y)z=R(x,y)z+\frac12Q(x,y)z
    -\frac14 B(x,y)z.
\end{equation}
Taking into account \eqref{tRRS}, \eqref{B=tA}, \eqref{tS=-Q} and \eqref{nP}, we establish the searched invariance.
\end{proof}

The following assertion is true due to \eqref{hRRSA}.
\begin{cor}\label{cor:inv.tensor2}
    The invariant tensor $K$ vanishes if and only if
    \[
    R(x,y)z=-\frac12 Q(x,y)z+\frac14 B(x,y)z.
    \]
\end{cor}

Let $A$ be the average tensor  of  $R$ and $\tR$, i.e.
$A(x,y)z=\frac12\{R(x,y)z+\tR(x,y)z\}$.
Then, because of \eqref{tRRS}, we have
\begin{equation}\label{bR=RS}
    A(x,y)z=R(x,y)z+\frac12 Q(x,y)z.
\end{equation}

\begin{prop}\label{prop:inv.tensor}
    The average tensor $A$ of $R$ and $\tR$ is an invariant tensor under the twin interchange, i.e. $A(x,y)z=\widetilde{A}(x,y)z$.
\end{prop}
\begin{proof}
Using \eqref{tRRS}, \eqref{tS=-Q} and \eqref{bR=RS},  we obtain the searched invariance.
\vskip-1em
\end{proof}

The following statement is an immediate consequence of \eqref{bR=RS}.
\begin{cor}\label{cor:bR=0}
    The invariant tensor $A$ vanishes if and only if $R=-\frac12Q$ is valid.
\end{cor}

From \eqref{hRRSA} and \eqref{bR=RS}, we have the following relation between the invariant tensors $K$, $A$ and $B$
\begin{equation}\label{wRbRPhi}
    K(x,y)z=A(x,y)z
    -\frac14 B(x,y)z.
\end{equation}

\begin{thm}\label{thm:inv.tensors}
    Any linear combination of the invariant tensors $A$ and $K$ is an invariant tensor under the twin interchange.
\end{thm}
\begin{proof}
It follows from \propref{prop:inv.tensor2} and \propref{prop:inv.tensor}.
\end{proof}

\section{Invariant connection and invariant tensors on the $\W$-manifolds in the main class}

Let us consider an arbitrary $\W$-manifold $(M,P,g)$ belonging to the basic class $\W_1$. Then we call that $(M,P,g)$ is a $\W_1$-manifold. This class is known as the main class in the classification in \cite{StaGri92}. The reason is that it is the only class where the fundamental tensor $F$ and the potential $\Phi$ are expressed explicitly by the metrics. Moreover, $\W_1$ has a special role with respect to conformal transformations of the metrics $g$ and $\tg$. 

Let us consider the conformal transformations $\widebar{g}=e^{2u}(\cosh{2v}\ g+\sinh{2v}\ \tg)$ of the $P$-metric $g$, where $u$ and $v$ are differentiable functions on the $\W$-manifold \cite{StaGri92}.
If $v=0$, we obtain the usual conformal transformation.
Then, the associated $P$-metric $\tg$ has the following image $\widebar{\tg}=e^{2u}(\cosh{2v}\ \tg+\sinh{2v}\ g)$.
Let us note that it is impossible for $g$ and $\tg$ to correspond one another through some conformal transformation.
%
According to \cite{StaGri92}, the class $\W_1$ is closed with respect to conformal transformations. Moreover, a $\W_1$-manifold is locally conformal equivalent to a $\W_0$-manifold  if and only if its Lee forms $\ta$ and $\ta^*$ are closed, i.e. $\D\ta=\D\ta^*=0$. In the latter case, the conformal transformations used are such that the 1-forms $\D u\circ P$ and $\D v\circ P$ are closed and it is said that the manifold belongs to the subclass $\W_1^0$ of $\W_1$.

Bearing in mind \thmref{thm:inv.cl} and using \eqref{tgP} and \eqref{tata*},
we obtain the following form of  $F$ of a $\W_1$-manifold 
under the twin interchange
\begin{equation*}\label{W1:tF}
\begin{array}{l}
  \tF(x,y,z)=-\frac{1}{2n}\bigl\{
  g(x,y)\ta(Pz)+g(x,z)\ta(Py)
  -g(x,Py)\ta(z)-g(x,Pz)\ta(y)\bigr\}.
\end{array}
\end{equation*}
Therefore, we get the following relation for a $\W_1$-manifold
\begin{equation*}\label{W1:tFF}
  \tF(x,y,z)=F(Px,y,z).
\end{equation*}

%
%

According to  \eqref{class2} and \eqref{hn=nP}, the invariant connection $D$ on a $\W_1$-manifold has the following form
\begin{equation*}\label{W1:inv-n}
D_x y=\n_x y+\frac{1}{4n}\left\{g(x,y)f^{\sharp}-\tg(x,y)Pf^{\sharp}\right\},
\end{equation*}
where $f^{\sharp}$ is the dual vector of the 1-form $f$ regarding $g$, i.e. $f(z)=g(f^{\sharp},z)$.

Since $\Phi$ has an explicit expression on a $\W_1$-manifold, i.e.
\begin{equation}\label{PhiW1}
\Phi(x,y)=\frac{1}{2n}\left\{g(x,y)f^{\sharp}-\tg(x,y)Pf^{\sharp}\right\},
\end{equation}
we find the concrete form of $Q$ and $B$ defined by \eqref{Q} and  \eqref{B}, respectively.

\begin{prop}\label{prop:W1_R}
If $(M,P,g)$ belongs to the class $\W_1$, then the tensors $Q$ and $B$ have the following form, respectively:
\begin{equation*}\label{W1:Q}
\begin{array}{l}
Q(x,y)z=\frac{1}{2n}\bigl\{g(y,z)Sx-g(x,z)Sy+\tg(y,z)S^*x-\tg(x,z)S^*y\bigr\},
\end{array}
\end{equation*}
\begin{equation*}\label{W1:B}
\begin{array}{l}
B(x,y)z=\frac{1}{4n^2}\bigl\{g(y,z)Hx-g(x,z)Hy
-\tg(y,z)HPx+\tg(x,z)HPy\bigr\},
\end{array}
\end{equation*}
where
\[
Sx=\n_x f^{\sharp}+\frac{1}{2n}f(x)f^{\sharp},\qquad
S^*x=\n_x Pf^{\sharp}+\frac{1}{2n}f(Px)f^{\sharp},
\]
\[
Hx=f(x)f^{\sharp}-f(Px)Pf^{\sharp}.
\]
\end{prop}
\begin{proof}
The formulae follow by direct computations, using \eqref{class}, \eqref{PhiW1}, \eqref{Q} and \eqref{B}.
\end{proof}

The expressions of $Q$ and $B$ in \propref{prop:W1_R}
 are substituted in the relations between
$R$ on the one hand and $\tR$, $K$, $A$ on the other,
given in \eqref{tRRS}, \eqref{hRRSA}, \eqref{bR=RS}, respectively.


\section{Lie group as a $\W_1$-manifold and its invariant objects under twin interchange}
\label{sec_3}

Let $L$ be a 4-dimensional real connected Lie group, and
$\mathfrak{l}$ be its Lie algebra with a basis
$\{X_{1},X_{2},X_{3},X_{4}\}$.

We introduce an almost paracomplex structure
$P$ and a $P$-metric $g$ by
\begin{equation}\label{Jdim4}
\begin{array}{llll}
PX_{1}=X_{2}, \quad & PX_{2}=X_{1}, \quad & PX_{3}=X_{4},
\quad &
PX_{4}=X_{3},
\end{array}
\end{equation}
\begin{equation}\label{g}
\begin{array}{c}
  g(X_1,X_1)=g(X_2,X_2)=-g(X_3,X_3)=-g(X_4,X_4)=1, \\[4pt]
  g(X_i,X_j)=0,\; i\neq j.
\end{array}
\end{equation}
Then, the associated $P$-metric $\tg$ of $g$  is determined by its non-zero components
\begin{equation}\label{tg}
\begin{array}{c}
  \tg(X_1,X_2)=-\tg(X_3,X_4)=1.
\end{array}
\end{equation}

Let us consider the constructed $\W$-manifold $(L,P,g)$ with the Lie algebra $\mathfrak{l}$
determined by the following nonzero commutators:
\begin{equation}\label{lie-w1-2}
\begin{array}{l}
\left[X_{1},X_{4}\right]=[X_{3},X_{2}]=\lm_{1}X_{1}+\e\lm_{1}X_{2}+\lm_{2}X_{3}+\e\lm_{2}X_{4},\\[4pt]
\left[X_{1},X_{3}\right]=[X_{4},X_{2}]=-\e\lm_{1}X_{1}-\lm_{1}X_{2}+\e\lm_{2}X_{3}+\lm_{2}X_{4},\\[4pt]
\left[X_{1},X_{2}\right]=2\lm_{2}X_{1}+2\e\lm_{2}X_{2},\qquad
\left[X_{3},X_{4}\right]=2\lm_{1}X_{3}+2\e\lm_{1}X_{4},
\end{array}
\end{equation}
where $\lm_1,\lm_2\in\R$ and $\e\in\{1,-1\}$.

Obviously for any $i,j\in\{1,2,3,4\}$, the identity $[P X_i,P X_j]=-[X_i,X_j]$ holds, i.e. $P$ is an Abelian structure for
$\mathfrak{l}$. This condition is equivalent to the equality $[P X_i, X_j]=-[X_i,P X_j]$. Then, according to  \eqref{NJ}, we obtain $N=0$ and $\wh{N}=0$.

\begin{thm}\label{thm:W10-W}
Let $(L,P,g)$ and $(L,P,\tg)$ be the pair of $\W$-manifolds, determined by
\eqref{Jdim4}--\eqref{lie-w1-2}. Then both the manifolds:
\begin{enumerate}\renewcommand{\labelenumi}{(\roman{enumi})}
    \item belong to the class $\W_1$ for arbitrary $\lm_1$ and $\lm_2$; 
		\item belong to the class of the locally conformal
        $\W_0$-manifolds for arbitrary $\lm_1$ and $\lm_2$;
		\item belong to the class of isotropic $\W_0$-manifolds if and only if $\lm_1=\pm\lm_2$;
    \item   are scalar flat if and only if $\lm_1=\pm\lm_2$;
		\item belong to $\W_0$ if and only if $\lm_1=\lm_2=0$. 
\end{enumerate}
\end{thm}
\begin{proof}
According to \eqref{tgP}, \eqref{Jdim4}, \eqref{g}, \eqref{lie-w1-2} and the Koszul equality for ($g$, $\n$) and ($\tg$, $\tn$), we obtain the following nonzero components of $\n$ and $\tn$:
\begin{equation}\label{nabla3}
\begin{array}{l}
\n_{X_1}X_1=\e\n_{X_2}X_1=-\e\tn_{X_1}X_2=-\tn_{X_2}X_2\\[4pt]
\phantom{\n_{X_1}X_1}=-2\lm_2X_2-\e\lm_1X_3+\lm_1X_4,\\[4pt]
\n_{X_1}X_2=\e\n_{X_2}X_2=-\e\tn_{X_1}X_1=-\tn_{X_2}X_1\\[4pt]
\phantom{\n_{X_1}X_2}=2\lm_2X_1-\lm_1X_3+\e\lm_1X_4,\\[4pt]
\n_{X_1}X_3=-\e\n_{X_1}X_4=\e\n_{X_2}X_3=-\n_{X_2}X_4\\[4pt]
\phantom{\n_{X_1}X_3}=\tn_{X_1}X_3=-\e\tn_{X_1}X_4=\e\tn_{X_2}X_3=-\tn_{X_2}X_4\\[4pt]
\phantom{\n_{X_1}X_3}=-\e\lm_1X_1-\lm_1X_2,\\[4pt]
\n_{X_3}X_1=-\e\n_{X_3}X_2=\e\n_{X_4}X_1=-\n_{X_4}X_2\\[4pt]
\phantom{\n_{X_1}X_3}
=\tn_{X_3}X_1=-\e\tn_{X_3}X_2=\e\tn_{X_4}X_1=-\tn_{X_4}X_2\\[4pt]
\phantom{\n_{X_1}X_3}
=-\e\lm_2X_3-\lm_2X_4,\\[4pt]
\n_{X_3}X_3=\e\n_{X_4}X_3=-\e\tn_{X_3}X_4=-\tn_{X_4}X_4\\[4pt]
\phantom{\n_{X_3}X_3}=-\e\lm_2X_1+\lm_2X_2-2\lm_1X_4,\\[4pt]
\n_{X_3}X_4=\e\n_{X_4}X_4=-\e\tn_{X_3}X_3=-\tn_{X_4}X_3\\[4pt]
\phantom{\n_{X_3}X_4}=-\lm_2X_1+\e\lm_2X_2+2\lm_1X_3.
\end{array}
\end{equation}
Using \eqref{tn=nPhi}, \eqref{Jdim4}, \eqref{g} and \eqref{nabla3},    we get the components $\Phi_{ij}=\Phi(X_i,X_j)$ of the anti-invariant tensor $\Phi$ as well as $f_{k}=f(X_k)$ and $f^*_{k}=f^*(X_k)$ of its associated 1-forms. The nonzero of them are the following
\begin{equation}\label{Phiij}
\begin{array}{l}
\frac14 f^\sharp=\Phi_{11}=\e\Phi_{12}=\e\Phi_{21}=\Phi_{22}=-\Phi_{33}=-\e\Phi_{34}=-\e\Phi_{43}=-\Phi_{44}\\[4pt]
\phantom{\frac14 f^\sharp}=-2\e\lm_2 X_1+2\lm_2 X_2+2\e\lm_1 X_3-2\lm_1 X_4,\\[4pt]
f_1=-\e f_2=\e f^*_1=-f^*_2=-8\e\lm_2,\\[4pt]
f_3=-\e f_4=\e f^*_3=-f^*_4=-8\e\lm_1.
\end{array}
\end{equation}

Taking into account \eqref{class2}, \eqref{Jdim4}, \eqref{g}, \eqref{tg} and \eqref{Phiij}, we obtain that $(L,P,g)$ belongs to $\W_1$. 
Since the class $\W_1$ is invariant under the twin interchange, it means that $(L,P,\tg)$ belongs to  $\W_1$, too.
This completes the proof of (i).

The correctness of (v) is a consequence of \eqref{Phiij} and \eqref{tP=-P}.

The components of $\n P$ and $\tn P$ follow from \eqref{nabla3} and \eqref{Jdim4}. 
Then, using \eqref{g}, \eqref{tg} and \eqref{F}, we
get the components
$F_{ijk}=F(X_{i},X_{j},X_{k})$ and $\tF_{ijk}=\tF(X_{i},X_{j},X_{k})=\tg((\tn_{X_i}P)X_j,X_k)$
of $F$ and $\tF$, respectively. The nonzero of them are determined by the following equalities and the identity 
$\tF_{ijk}=\e F_{ijk}$ that holds for the constructed manifolds $(L,P,g)$ and $(L,P,\tg)$
\begin{equation}\label{lambdi}
\begin{array}{l}
2\lm_{1}=F_{113}=-F_{124}=F_{131}=-F_{142}=-\e F_{114}=\e F_{123}=\e F_{132}\\[4pt]
\phantom{2\lm_{1}}
=-\e F_{141}=\e F_{213}=-\e F_{224}=\e F_{231}=-\e F_{242}=-F_{214}= F_{223}\\[4pt]
\phantom{2\lm_{1}}
=F_{232}=- F_{241}
=-\frac12 F_{333}=\frac12 F_{344}=-\frac12\e F_{433}=\frac12\e F_{444},\\[4pt]
2\lm_{2}=-F_{313}=F_{324}=-F_{331}=F_{342}=-\e F_{314}=\e F_{323}=\e F_{332}\\[4pt]
\phantom{2\lm_{1}}
=-\e F_{341}=-\e F_{413}=\e F_{424}=-\e F_{431}=\e F_{442}=-F_{414}= F_{423}\\[4pt]
\phantom{2\lm_{1}}
=F_{432}=- F_{441}
=\frac12 F_{111}=-\frac12 F_{122}=\frac12\e F_{211}=-\frac12\e F_{222}.
\end{array}
\end{equation}

Applying \eqref{snorm} for the components of $F$ and $\tF$, we obtain the square norms of $\n
P$ and $\tn P$ as follows
\begin{equation}\label{nJ}
\nP=-\e \tnP=-128\left(\lm_1^2-\lm_2^2\right).
\end{equation}
Then, the latter equalities imply the statement (iii).

Taking into account \eqref{tata*}, we have
$\ta_k=\widetilde{\ta}_k$ and
$\ta^*_k=\widetilde{\ta}^*_k$ for the corresponding components with respect to $X_k$.
Furthermore, the same situation is for  $\D{\ta}=\D\widetilde{\ta}$ and $\D{\ta}^*=\D\widetilde{\ta}^*$.
By \eqref{ta}, \eqref{ta*taJ} and \eqref{lambdi}, we obtain
$\ta_k$ and $\ta^*_k$ and thus we get
\begin{equation}\label{theta123}
\begin{array}{ll}
\ta_{1}=-\e\ta_{2}=\e\ta^*_{1}=-\ta^*_{2}
=\widetilde{\ta}_{1}=-\e\widetilde{\ta}_{2}=\e\widetilde{\ta}^*_{1}=-\widetilde{\ta}^*_{2}=8\lm_{2}, \\[4pt]
\ta_{3}=-\e\ta_{4}=\e\ta^*_{3}=-\ta^*_{4}
=\widetilde{\ta}_{3}=-\e\widetilde{\ta}_{4}=\e\widetilde{\ta}^*_{3}=-\widetilde{\ta}^*_{4}=8\lm_{1}.
\end{array}
\end{equation}
Using \eqref{lie-w1-2} and \eqref{theta123}, we compute the components of
$\D\ta$ and $\D\ta^*$ with respect to the basis
$\{X_{1},X_{2},X_{3},X_{4}\}$. We obtain that the 1-forms $\ta$, $\ta^*$, $\widetilde\ta$, $\widetilde{\ta}^*$ are closed un\-con\-di\-tio\-nal\-ly, i.e. the statement (ii) holds.

By virtue of \eqref{g}, \eqref{lie-w1-2} and \eqref{nabla3}, we get  $R_{ijkl}=R(X_{i},X_{j},X_{k},X_{l})$ and $\tR_{ijkl}=\tR(X_{i},\allowbreak{}X_{j},\allowbreak{}X_{k},X_{l})$, the basic components of the curvature tensors
for $\n$ and $\tn$. The nonzero ones of them are determined by \eqref{curv} and the following:
\begin{equation}\label{R3}
\begin{array}{l}
R_{1221} = -2\e R_{1341} =-2\e R_{2342} = -8\lm_{2}^{2},   \\[4pt]
R_{1331} = R_{1441} =R_{2332} = R_{2442} = 4(\lm_{2}^{2} - \lm_{1}^{2}), \\[4pt]
R_{3443} = -2\e R_{3123} =-2\e R_{4124} = 8\lm_{1}^{2},\\[4pt]
R_{1241} =R_{2132} =-R_{3243}=-R_{4134} = -4\lm_{1}\lm_{2},\\[4pt]
R_{1231} =R_{2142} =-R_{3143}=-R_{4234} = 4\e\lm_{1}\lm_{2},
\end{array}
\end{equation}
\begin{equation}\label{tR3}
\begin{array}{l}
\tR_{ijkl}=\e R_{ijkl}.
\end{array}
\end{equation}
Therefore, the components of the Ricci tensors
and the values of the scalar curvatures
for $\n$ and $\tn$ are:
\begin{equation}\label{Ricci3}
\begin{array}{c}
\begin{array}{ll}
\rho_{11}=\rho_{22}=8\big(\lm_{1}^{2} -2 \lm_{2}^{2}\big), \qquad &
\widetilde{\rho}_{11}=\widetilde{\rho}_{22}=-\e\widetilde{\rho}_{12}=-8\lm_2^2,
\\[4pt]
\rho_{33}=\rho_{44}=8\big(\lm_{2}^{2} -2 \lm_{1}^{2}\big), \qquad &
\widetilde{\rho}_{33}=\widetilde{\rho}_{44}=-\e\widetilde{\rho}_{34}=-8\lm_1^2,\\[4pt]
\end{array}
\\[4pt]
\begin{array}{l}
\rho_{13}=\rho_{24}=-\e\rho_{14}=-\e\rho_{23}=\widetilde{\rho}_{13}=\widetilde{\rho}_{24}=-\e\widetilde{\rho}_{14}=-\e\widetilde{\rho}_{23}=-8\lm_{1}\lm_{2},
\end{array}
\\[4pt]
\begin{array}{lll}
\rho_{12}=\rho_{34}=0, \qquad &
\tau=48\big(\lm_{1}^{2}-\lm_{2}^{2}\big),\qquad &
\widetilde{\tau}=16\e\big(\lm_{2}^{2}-\lm_{1}^{2}\big).
\end{array}
\end{array}
\end{equation}


The truthfulness of (iv) follows immediately from the last two equations in \eqref{Ricci3}. This completes the proof. %
\end{proof}


\subsection{The invariant connection and invariant tensors under the twin interchange }

By virtue of \eqref{B} and \eqref{Phiij}, we establish that all basic components $B_{ijk}=B(X_i,X_j)X_k$ of the invariant tensor $B$ are zero and thus $B=0$ holds.

After that, we take in account \eqref{Q} and \eqref{Phiij} to calculate the basic components $Q_{ijk}=Q(X_i,X_j)X_k$ of the anti-invariant tensor $Q$. The nonzero of them are the following and the rest are determined by the property $Q_{ijk}=-Q_{jik}$:
\begin{equation}\label{Qij}
\begin{array}{l}
Q_{131} = -Q_{142} = Q_{232} = -Q_{241} = \frac12 Q_{344}=\frac12 \e \lm_{1}f^\sharp, \\[4pt]
Q_{132} = -Q_{141} = Q_{231} = -Q_{242} =\frac12 Q_{343}= \frac12 \lm_{1}f^\sharp, \\[4pt]
Q_{133} = Q_{144} = -Q_{234} = -Q_{243} =-\frac12 Q_{122}= \frac12 \e \lm_{2}f^\sharp, \\[4pt]
Q_{134} = Q_{143} = -Q_{233} = -Q_{244} =-\frac12 Q_{121} = \frac12 \lm_{2}f^\sharp.
\end{array}
\end{equation}

We compute the basic components $A_{ijk}=A(X_i,X_j)X_k$ of the invariant tensor $A$, using \eqref{R3}, \eqref{tR3} and that this tensor is the average tensor of $R$ and $\tR$. Thus we get the components $A_{ijkl}=g\left(A(X_i,X_j)X_k,X_l\right)$:
\begin{equation*} \label{Aijkl}
\begin{array}{l}
\begin{array}{l}
2\lm_{1}^{2} = A_{1324} = A_{1423} = A_{2314} = A_{2413} = -\frac12 A_{3434} =\frac12 A_{3443}\\[4pt]
\phantom{2\lm_{1}^{2}}
=\e A_{1323} = \e A_{1424} = \e A_{2313} = \e A_{2414} = -\frac12\e A_{3433} =\frac12\e A_{3444},\\[4pt]
2\lm_{2}^{2} = A_{1342} = A_{1432} = A_{2341} = A_{2431} = \frac12 A_{1212} =-\frac12 A_{1221}\\[4pt]
\phantom{2\lm_{2}^{2}}=\e A_{1341} = \e A_{1431} = \e A_{2342} = \e A_{2431} = 
\frac12\e A_{1211} =-\frac12\e A_{1222},\\[4pt]
\end{array}\\
\begin{array}{l}
2\lm_{1}\lm_{2} =\frac12 A_{1232}=-\frac12 A_{1241}=-\frac12 A_{3414}=\frac12 A_{3423}= -A_{1311} = A_{1322} 
\\[4pt]
\phantom{2\lm_{1}\lm_{2} }
=-A_{1333} =A_{1344} = A_{1412}= -A_{1421} =-A_{1434}=A_{1443}   
 \\[4pt]
\phantom{2\lm_{1}\lm_{2} }
 = -A_{2312}= A_{2321} =A_{2334} =-A_{2343}= A_{2411} = -A_{2422}   \\[4pt]
\phantom{2\lm_{1}\lm_{2} }
=A_{2433}=-A_{2444}=\frac12 \e A_{1231}=-\frac12 \e A_{1242}=-\frac12 \e A_{3413} \\[4pt]
\phantom{2\lm_{1}\lm_{2} }
=\frac12 \e A_{3424}= -\e A_{1312} = \e A_{1321}=-\e A_{1334} =\e A_{1343}= \e A_{1411}    
 \\[4pt]
\phantom{2\lm_{1}\lm_{2} }
= -\e A_{1422}=-\e A_{1433}=\e A_{1444} = -\e A_{2311}= \e A_{2322} =\e A_{2333}    \\[4pt]
\phantom{2\lm_{1}\lm_{2} }
=-\e A_{2344}= \e A_{2412}= -\e A_{2421} =\e A_{2434} =-\e A_{2443}, 
\\[4pt]
2\lm_{1}^2-4\lm_{2}^2 =A_{1313}= A_{1414}= A_{2323}= A_{2424} \\[4pt]
\phantom{2\lm_{1}^2-4\lm_{2}^2 }
= \e A_{1314} = \e A_{1413} = \e A_{2324} = \e A_{2423}\\[4pt]
2\lm_{2}^2-4\lm_{1}^2 =A_{1331}= A_{1441}= A_{2332}= A_{2442} \\[4pt]
\phantom{2\lm_{1}^2-4\lm_{2}^2 }
= \e A_{1332} = \e A_{1442} = \e A_{2331} = \e A_{2441}.
\end{array}
\end{array}
\end{equation*}
The rest components are determined by property $A_{ijkl}=-A_{jikl}$. Let us remark that $A$ is not a curvature-like tensor.

Obviously, $A=0$ if and only if the corresponding Lie algebra is Abelian and $(L,P,g)$ is a $\W_0$-manifold.

The Nijenhuis tensors $N$ and $\widetilde{N}$  on $(L,P,g)$ and $(L,P,\tg)$, respectively, vanish as on any $\W_1$-manifold.
According to \cite{StaGri92}, the condition $N=0$ is equivalent to the property  $\Phi(X_i,X_j)=\Phi(PX_i,PX_j)$. Then, by means of \eqref{wNPhi}, we obtain for the components of the associated Nijenhuis tensor  $\widetilde{\wh{N}}$ (similarly for $\wh{N}$) expressed by the components of the potential $\Phi$, given in \eqref{Phiij}, as follows
\[
\wh{N}_{ij}=-4\Phi_{ij},\qquad \widetilde{\wh{N}}_{ij}=-4\widetilde{\Phi}_{ij}.
\]
Let us recall that the tensors $\wh{N}_{ij}$ and $\Phi_{ij}$ are anti-invariant under twin interchange.

Bearing in mind \eqref{hn=nP}, \eqref{nabla3} and \eqref{Phiij}, we get the components of the invariant connection as follows
\begin{equation}\label{hn-ex}
\begin{array}{l}
D_{X_{1}}X_{1} = -\e D_{X_{1}}X_{2} = \e D_{X_{2}}X_{1} = -D_{X_{2}}X_{2} =
 -\e\lm_{2}X_{1}-\lm_{2}X_{2},\\[4pt]
D_{X_{1}}X_{3} = -\e D_{X_{1}}X_{4} = \e D_{X_{2}}X_{3} = -D_{X_{2}}X_{4}  =
 -\e\lm_{1}X_{1}-\lm_{1}X_{2},\\[4pt]
D_{X_{3}}X_{1} = -\e D_{X_{3}}X_{2} = \e D_{X_{4}}X_{1} = -D_{X_{4}}X_{2} =
 -\e\lm_{2}X_{3}-\lm_{2}X_{4},\\[4pt]
D_{X_{3}}X_{3} = -\e D_{X_{3}}X_{4} = \e D_{X_{4}}X_{3} = -D_{X_{4}}X_{4}  =
 -\e\lm_{1}X_{3}-\lm_{1}X_{4}.
\end{array}
\end{equation}

Taking into account \eqref{wRbRPhi} and the vanishing of $B$, we obtain that the invariant ten\-sors $K$ and $A$ for $(L,P,g)$ coincide. Similarly, we find the corresponding property for $(L,P,\tg)$. So, we have the equalities 
\[
K=A,\qquad \widetilde{K}=\widetilde{A}.
\]
A way to check the latter equalities is a direct computation of the basic components $K_{ijk}=K(X_i,X_j)X_k$ of $K$ from \eqref{hn-ex} as the curvature tensor of the invariant connection $D$.



\begin{thebibliography}{99}

\bibitem{BoFeFrVo99}
A. Borowiec, M. Ferraris, M. Francaviglia, I. Volovich.
\textit{Almost complex and almost product Einstein manifolds from a variational principle.} 
J. Math. Phys. \textbf{40} (1999), 446--464.

\bibitem{Dok05}
I. Dokuzova. 
\textit{Some properties of one connection on spaces with an almost product structure.} 
J. Tech. Univ. Plovdiv, Fundam. Sci. Appl. \textbf{11} (2005-2006), 21--27.

\bibitem{GriManMek06}
K. Gribachev, M. Manev, D. Mekerov. 
\textit{A Lie group as a 4-dimensional quasi-K\"ahler manifold with Norden metric.} 
JP J. Geom. Topol.  \textbf{6} (2006), no. 1, 55--68.

\bibitem{DGri12}
D. Gribacheva.
\textit{A natural connection on a basic class of Riemannian product manifolds.}
Int. J. Geom. Methods Mod. Phys.
\textbf{9} (2012), no. 7, 1250057 (14 pages).


\bibitem{DGriMek11}
D. Gribacheva, D. Mekerov. 
\textit{Canonical connection on a class of Riemannian almost product manifolds.} 
J. Geom. \textbf{102} (2011), 53--71.


\bibitem{HMMek15}
H. Manev, D. Mekerov. 
\textit{Lie groups as 3-dimensional almost contact B-metric manifolds}. 
J. Geom. \textbf{106} (2015), 229--242.


\bibitem{Man49}
M. Manev, 
\textit{Invariant tensors under the twin interchange of Norden metrics on almost complex manifolds}. 
Results Math. \textbf{70} (2016), no. 1, 109--126.


\bibitem{ManTav19}
M. Manev, V. Tavkova, \textit{On almost paracontact almost paracomplex Riemannian manifolds.}
arXiv:1805.11120.


\bibitem{Mek08}
D. Mekerov, 
\textit{On Riemannian almost product manifolds with nonintegrable structure.} 
J. Geom. \textbf{89} (2008), no. 1-2, 119--129.


\bibitem{MekMan12}
D. Mekerov, M. Manev,
\textit{Natural connection with totally skew-symmetric torsion on Riemannian almost
product manifolds.}
Int. J. Geom. Methods Mod. Phys. \textbf{9} (2012), no. 1, 1250003 (14 pages).


\bibitem{Mih87}
V. Mihova, 
\textit{Canonical connection and the canonical conformal group on a Riemannian
al\-most-product manifold}. 
Serdica Math. J. \textbf{15} (1987), 351--358.

\bibitem{Nav}
A. M. Naveira. 
\textit{A classification of Riemannian almost product manifolds.}
Rend. Math. Roma \textbf{3} (1983), 577--592.

\bibitem{Pri05}
G. T. Pripoae. 
\textit{A sharper classification of semi-riemannian almost product manifolds.}
Tensor  (N.S.) \textbf{66} (2005), no. 1, 9--17.


\bibitem{Sta87}
M. Staikova, 
\textit{Curvature properties of Riemannian P-manifolds.} 
Plovdiv Univ. Sci. Works -- Math. \textbf{32} (1987), no. 3, 241--251.


\bibitem{StaGri92}
M. T. Staikova, K. I. Gribachev,
\textit{Canonical connections and their conformal invariants on Riemannian almost-product manifolds.}
Serdica Math. J. \textbf{18} (1992), 150--161.

\end{thebibliography}
\end{document}